\documentclass[11pt]{article}

\usepackage[a4paper,margin=1.15in]{geometry}
\usepackage{amsmath,amssymb,amsthm,mathtools}
\usepackage{hyperref}
\usepackage{mathrsfs}
\usepackage{tikz-cd}
\usetikzlibrary{decorations.markings, arrows.meta}
\usetikzlibrary{fit}

\hypersetup{
  colorlinks=true,
  linkcolor=blue,
  citecolor=blue,
  urlcolor=blue
}

\usepackage{xcolor}

\newtheorem{theorem}{Theorem}[section]
\newtheorem{proposition}[theorem]{Proposition}
\newtheorem{lemma}[theorem]{Lemma}

\theoremstyle{definition}
\newtheorem{definition}[theorem]{Definition}
\newtheorem{remark}[theorem]{Remark}
\newtheorem*{definition*}{Definition}
\newtheorem*{lemma*}{Lemma}

\newcommand{\R}{\mathbb{R}}

\newcommand{\Sig}{\operatorname{Sig}}

\newcommand{\id}{\mathbf{1}}

\newcommand{\proj}[1]{\pi_{#1}}
\newcommand{\wind}{\operatorname{Wind}}

\title{An Elementary Proof of the Hambly--Lyons \\ Uniqueness Theorem}
\author{Josef Teichmann, Walter Schachermayer, Valentin Tissot-Daguette}
\date{}

\author{
Walter Schachermayer\thanks{
  Faculty of Mathematics, University of Vienna, walter.schachermayer@univie.ac.at} \hspace{1em}
  Josef Teichmann\thanks{
  Department of Mathematics, ETH Zurich, jteichma@math.ethz.ch} \hspace{1em}
Valentin Tissot-Daguette\thanks{
  Quantitative Research, Office of the CTO, Bloomberg,  vtissotdague@bloomberg.net}  
}
\date{\today}

\begin{document}

\maketitle

\begin{abstract}
We give a self-contained proof, in the bounded variation setting, of the Hambly--Lyons uniqueness theorem, which states that (total) signature identifies the path up to tree-like equivalences. The argument is organized around two key 
 geometric observations. First, tree-like paths have trivial signature because factorization over a loop in a tree $\tau:[0,1]\to T$ is preserved under signature lifts, which follows from an elementary property of planar curves. 
Second, a path with trivial total signature contains a nontrivial subpath with trivial total signature (the sub-interval lemma). This is proven by applying a winding-number argument to a two-dimensional projection of the signature lift.  
Collapsing all trivial-signature sub-intervals then defines a compact metric tree $T$ through which the original path factors by virtue of the sub-interval lemma.
\end{abstract}

\section{Introduction}

Let \(\gamma:[0,1]\to\R^d\) be a continuous path of bounded variation. Its signature records all iterated integrals of \(\gamma\) over $[0,1]$ (see Section \ref{sec:signatures}). The signature is invariant under (monotone increasing) reparametrizations of $[0,1]$ and shifts in $\mathbb{R}^d$, and does not change 
if tree-like excursions are attached to the paths. If we fix the starting point of $\gamma$, then the corresponding geometric equivalence relation is tree-like equivalence (see Section \ref{sec:tree-like}), as established in  the  seminal work of Ben Hambly and Terry Lyons, see \cite{HL2010}. The theorem says precisely that for continuous bounded variation curves with $\gamma(0)=0$,  
it holds that
\[
  \Sig_{0,1}(\gamma)= \id
  \quad\Longleftrightarrow\quad
  \gamma \text{ is tree-like} \, .
\]
The corresponding assertion for Banach space valued rough paths has been proven in the impressive work \cite{BGLY2016}. It is the goal of this article to provide an argument that makes the challenging direction from trivial signature to tree-like elementary in full generality.

Signatures also constitute the complete list of reparametrization-invariant polynomials on the Banach space of bounded variation curves, which will be proven elsewhere. This allows us to draw conclusions from properties of well-known families of real analytic functions on the Banach space of bounded variation curves, like winding numbers around points, on the signatures of the underlying curves. This lies at the core of our proof.

We first recall the extended tensor algebra and the time-indexed signature \(\Sig_{s,t}(\gamma)\) of a path $\gamma$. We then introduce tree-like paths. The implication from tree-like paths to trivial signature is proved by induction. 
For the converse, we use a winding-number argument on a two-dimensional projection of the signature curve
\[
  t\longmapsto \Sig_{0,t}(\gamma),
\]
followed by a quotient construction of the desired tree.

\section{Extended tensor algebra and signatures}\label{sec:signatures}

The extended tensor algebra over $\R^d$ is defined
\[
  T((\R^d)):=\prod_{n=0}^{\infty} (\R^d)^{\otimes n},
  \qquad (\R^d)^{\otimes 0}:=\R \, .
\]
An element $ a\in T((\R^d)) $ is therefore given by a sequence
\[
  a=(a^{(0)},a^{(1)},a^{(2)},\ldots),
  \qquad a^{(n)}\in (\R^d)^{\otimes n} \, ,
\]
where $T((\mathbb{R}^d))$ inherits the structure of a locally convex topological vector space from $\mathbb{R}^\mathbb{N}$.
Multiplication is the Cauchy tensor product
\[
  (ab)^{(n)}:=\sum_{k=0}^n a^{(k)}\otimes b^{(n-k)}.
\]
The unit is defined as
\[
\id  :=(1,0,0,\ldots).
\]
For $N\geq 0$, we write
\[
  T^{(N)}(\R^d):=\prod_{n=0}^N (\R^d)^{\otimes n}
\]
for the truncated tensor algebra, and denote
\[
  \proj{N}:T((\R^d))\to T^{(N)}(\R^d)
\]
for the canonical projection. We shall use the same symbol $\id$ for the unit in any truncation.

Let $e_1,\ldots,e_d$ be the canonical basis of $\R^d$. A word $w=i_1\cdots i_n$ is a tuple (written as an ordered product here) identified with the basis tensor in $(\mathbb{R}^d)^{\otimes n}$ and is conveniently written as
\[
e_w = e_{i_1}\cdots e_{i_n}  :=e_{i_1}\otimes\cdots\otimes e_{i_n}.
\]
The corresponding coordinate projection of $a\in T((\R^d))$ is denoted by $a^w \in \mathbb{R}$, leading more generally to the map  
$$
.^w : T((\mathbb{R}^d)) \to \mathbb{R} \, .
$$
We furthermore write  $ T(\mathbb{R}^d) $ for the space of linear combinations of all coordinate projections $ a \mapsto a^w $ over all words $ w $, which is the topological dual space of the locally convex vector space $T((\mathbb{R}^d))$. It carries the locally convex topology of $\mathbb{R}^{(\mathbb{N})}$.

Let $\gamma:[0,1]\to\R^d$ be continuous and of bounded variation. For every sub-interval $[s,t]\subset[0,1]$, $s\leq t$, and every word $w=i_1\cdots i_n$, we define the signature component associated to $w$ as
\[
  S^w_{s,t}(\gamma)
  :=
  \int_{s<u_1<\cdots<u_n<t}
  d\gamma^{i_1}_{u_1}\cdots d\gamma^{i_n}_{u_n} \, .
\]
For the empty word \(\emptyset\), we set \(S^{\emptyset}_{s,t}(\gamma)=1\). The time-indexed signature over \([s,t]\) is then
\[
  \Sig_{s,t}(\gamma)
  :=
  \sum_w S^w_{s,t}(\gamma)e_w
  =
  \id +
  \sum_{n\geq 1}\sum_{i_1,\ldots,i_n=1}^d
  S^{i_1\cdots i_n}_{s,t}(\gamma)e_{i_1}\cdots e_{i_n}
  \in T((\R^d)) \, ,
\]
for $s \leq t$, as an element of the extended tensor algebra $T((\mathbb{R}^d))$. We write
\[
  \Sig^{\leq N}_{s,t}(\gamma):=\proj{N}\Sig_{s,t}(\gamma)
\]
for the truncated (at level $N$) signature and call
\[
t \mapsto  \Gamma_t:=\Sig_{0,t}(\gamma) \in T((\mathbb{R}^d))
\]
the \emph{signature curve} of $\gamma$ on the interval $[0,1]$. It is a continuous $T((\R^d))$-valued curve, which connects $\id$ and $\Sig_{0,1}(\gamma)$. As an example, the first level satisfies
\[
  S^i_{s,t}(\gamma)=\gamma^i_t-\gamma^i_s \, .
\]
In the special case $d=2$ and $\gamma(0)=\gamma(1)$, the anti-symmetric part of the second level is the signed area swept out by $ t \mapsto \gamma_t $
\begin{equation}\label{eq:LevyArea}
      A(\gamma)
  :=
  \frac12\left(\int_0^1 \gamma^1_u\,d\gamma^2_u
  -\int_0^1 \gamma^2_u\,d\gamma^1_u\right) \, ,
\end{equation}
otherwise called \emph{Lévy area}.

We call $ \Sig_{0,1}(\gamma)$ the (total) signature of the curve $\gamma$. In particular, it holds that $\Gamma$ connects $\id$ and the total signature in $T((\R^d))$, i.e., $\Gamma_1=\Sig_{0,1}(\gamma)$. The natural question arises concerning which properties of $\gamma$ are characterized by its total signature. A complete answer has been given in \cite{HL2010}.

We shall use several basic properties of signatures of continuous bounded variation curves starting at $0$:

\begin{proposition}[Chen's identity \cite{C1954}]\label{prop:chen}
For $ 0\leq r, s, t \leq 1$,
\[
  \Sig_{r,t}(\gamma)=\Sig_{r,s}(\gamma)\Sig_{s,t}(\gamma).
\]
In particular, every $\Sig_{s,t}(\gamma)$ is invertible in $T((\mathbb{R}^d))$. For $ s,t \in [0,1]$, we have
\[
  \Sig_{s,t}(\gamma):=\Sig_{t,s}(\gamma)^{-1} \, .
\]
Chen's identity also holds for truncated signatures of all orders.
\end{proposition}

\begin{remark}
Chen's identity follows from the fact that $\Sig_{s,t}(\gamma)$ for $ s, t \in [0,1]$ is the unique evolution associated to the solution of the ordinary differential equation
$$
d x_u = \sum_{i=1}^d x_u e_i d\gamma^i_u, \; \;   x_s = \id  \;,
$$
on $T((\mathbb{R}^d))$, and  therefore satisfies a semigroup property, which is precisely Chen's identity. The same argument applies to truncated signatures.
\end{remark}

Furthermore, the signature is obviously reparametrization-invariant:

\begin{proposition}[reparametrization invariance]\label{prop:reparam}
Fix $ s < t $ in $[0,1]$. Let $\phi:[0,1]\to[s,t]$ be continuous, increasing and onto. Then
\[
  \Sig_{0,1}(\gamma\circ\phi)=\Sig_{s,t}(\gamma) \, .
\]
\end{proposition}

\begin{proposition}\label{prop:SW}
Consider the continuous functions $C(\Gamma([0,1]))$ on the range 
\begin{equation}\label{eq:SigRange}
    \Gamma([0,1]) = \{\Gamma_t, \ t \in [0,1]\}, 
\end{equation}
which is a compact subset of the locally convex space $T((\R^d))$. Then the restrictions of linear functionals, i.e., the restrictions of elements of $T(\mathbb{R}^d)$ on $\Gamma([0,1])$,  form a point separating sub-algebra and are therefore dense in $C(\Gamma([0,1]))$ with respect to the uniform topology. 
\end{proposition}

\begin{proof}
Since products of signature components are linear combinations of signature components by the Leibniz rule, the restriction of linear functionals on $\Gamma([0,1])$ forms an algebra. This sub-algebra of the algebra of continuous functions is furthermore point separating since the dual space $T(\R^d)$ is point separating as well. Therefore, we can conclude by the Stone--Weierstrass theorem. 
\end{proof}

\begin{lemma}[Linear one-form lemma]\label{lem:polynomial-one-form}
Let $ \ell_1, \ell_2 \in T(\mathbb{R}^d) $ be two linear functionals on $T((\mathbb{R}^d))$. Then
\[
  \int_0^1 \ell_1(\Gamma_t)\,d \ell_2(\Gamma_t)
\]
is a finite linear combination of total signature components of $ \gamma $, i.e., there exists $ \ell \in T(\mathbb{R}^d)$ such that the above integral equals $\ell(\Gamma_1)$. Clearly the signature component of the empty word does not appear in this linear combination.
\end{lemma}

\begin{proof}
Each $w$-coordinate of $\Gamma_t$, i.e., $\Gamma_t^w = S^w_{0,t}(\gamma) \in \R$, $w = i_1\cdots i_n$,
is an iterated integral of $ \gamma $ over $[0,t] $. Since products of iterated integrals are linear combinations of iterated integrals by the Leibniz rule, the integration of an iterated integral against an iterated integral gives again a finite linear combination of iterated integrals of $ \gamma $ over $ [0,1] $. Hence the value of the integral is a finite linear combination of total signature components.
\end{proof}


In the planar case ($d = 2$), the following proposition relates the Lévy area to the winding numbers of $\gamma$, defined as  
$$\textnormal{Wind}_{z_0}(\gamma) = \frac{1}{2 \pi \mathrm{i}} \int_{\gamma} \frac{1}{z-z_0} d z , \quad z_0 \in \mathbb{C} \setminus \gamma([0,1]) \, .$$
\begin{proposition}\label{prop:levyarea_winding}
Let $d=2$ and $\gamma(0)=\gamma(1)$. Then the Lévy area \eqref{eq:LevyArea} can be expressed as  
\begin{equation}\label{eq:LevyvsWinding}
    A(\gamma) = \int_{\mathbb{C} \setminus \gamma([0,1])}  \textnormal{Wind}_{z_0}(\gamma) \, \lambda(dz_0),
\end{equation}
where $\lambda$ denotes the planar Lebesgue measure. 
\end{proposition}
\begin{proof}
This follows from the more general Cufí--Verdera  formula  \cite[Theorem 1]{CV2015}, or can simply be seen directly by Fubini's theorem and the well-known potential-theoretic formula
$$
\frac{1}{\pi} \int_{B_R} \frac{z-z_0}{{|z-z_0|}^2} \lambda(dz_0)  = z
$$
for a disc $B_R \subset \mathbb{C}$ around zero with radius $ R > |z|$.
\end{proof}

\section{Trees and tree-like paths}\label{sec:tree-like}

Recall that a closed curve on a finite tree splits into elementary excursions: whenever the walk enters a branch, it must later leave that branch along the same arc. General compact trees are treated with the same intuition. 

\begin{definition}[Compact tree]
A compact tree is a compact connected metric space $T$ such that any two points $a,b\in T$ are connected by a unique arc, i.e., a unique, injective continuous curve. This arc is denoted by $[a,b]$.
\end{definition}

\begin{definition}[Tree-like path]
A continuous bounded variation path $ \gamma:[0,1]\to\R^d $ is called tree-like if there exist a compact tree $ T $, a continuous loop
\[
  \tau:[0,1]\to T,
  \qquad \tau(0)=\tau(1),
\]
and a continuous map $ \psi:T\to\R^d $ such that
\[
  \gamma=\psi\circ\tau.
\]
In particular, $ \gamma(1)=\gamma(0) $. Analogously, we speak of tree-like paths on $[s,t]$.
\end{definition}

Tree-like paths can also be described by so-called height functions \cite{HL2010}. Here, we adopt an equivalent definition following \cite{BGLY2016}, which better fits our geometric viewpoint. See also \cite[Theorem 5.15]{L2017} for other equivalent definitions of tree-like paths.

\begin{remark}\label{openmap}
If a continuous bounded variation path $ \gamma:[0,1]\to\R^d $ factors over a continuous loop $ \tau:[0,1] \to T$ through $ \gamma= \psi \circ \tau$, then $\psi|_{\tau([0,1])}$ is automatically continuous.
\end{remark}

We now record the elementary two-dimensional lemma used in the subsequent induction over signature levels.

\begin{lemma}[Tree-like planar curves have vanishing Lévy area]\label{lem:2DTreeLikeLemma}
Let $\beta:[0,1]\to\R^2$ be a continuous bounded variation path. If $\beta$ is tree-like, then
\[
  \Sig^{\leq 2}_{0,1}(\beta)=\id \, .
\]
Equivalently, $\beta$ is a loop  with vanishing Lévy area: $A(\beta) = 0$. 
\end{lemma}

\begin{proof}
  Since $\beta$ is tree-like, it factors as $\beta=\psi\circ\tau$, where $\tau:[0,1] \to T$ is a loop in a compact tree. Hence $\beta(1)=\beta(0)$, and the first level of the signature vanishes.

It remains to consider the second level. We now interpret $\beta$ as a curve in $\mathbb{C}$. Let $ z_0 \notin \beta([0,1]) $. By Cauchy's theorem of complex analysis, $\int_\beta \frac{1}{z-z_0} d z=0$, since $\beta$ can be retracted to a point within its own range by being tree-like \cite[Theorem 5.15]{L2017}. On the other hand, we have for all closed, continuous bounded variation curves
\[
2 \pi \textrm{i} A(\beta) = \int_{\mathbb{C} \setminus \beta([0,1])} \int_{\beta} \frac{1}{z-z_0} d z \, \lambda(dz_0) \, ,
\]  
by Proposition \ref{prop:levyarea_winding}. Hence the L\'evy area $A(\beta)$ vanishes.

The symmetric part of the second level is determined by the identity:
\[
   \Sig^{ij}_{0,1}(\beta)+  \Sig^{ji}_{0,1}(\beta)
  =
  \Sig^i_{0,1}(\beta)  \Sig^j_{0,1}(\beta)=0 \, .
\]
Together with the vanishing of the anti-symmetric part, this gives $\Sig^{ij}_{0,1}(\beta)=0$ for all $i,j\in\{1,2\}$. Hence $\Sig^{\leq 2}_{0,1}(\beta)=\id$.
\end{proof}

\begin{proposition}[Tree-likeness of signature curves]\label{prop:lifts-tree-like}
If $ \gamma:[0,1]\to\R^d $ is tree-like, i.e., it factors over a continuous loop $ \tau:[0,1] \to T $ for some compact tree $T$, then for every $N\geq 1$, the truncated signature curve
\[
  t\longmapsto \Sig^{\leq N}_{0,t}(\gamma)
\]
factors over the same loop $\tau: [0,1] \to T$.
\end{proposition}

\begin{proof}
We argue by induction on $N$. For $N=1$ the lifted path is $t\mapsto \gamma(t)-\gamma(0)$. Thus it is tree-like by assumption and factors over the loop $\tau:[0,1] \to T $.

Assume that
\[
  X_t:=\Sig^{\leq N}_{0,t}(\gamma)
\]
is tree-like factoring over the same loop $ \tau:[0,1] \to T$ as $\gamma$ does. Fix $s<t$ such that $\tau(s)=\tau(t)$; then $\tau|_{[s,t]}$ is a loop in $T$, too. The increment path
\[
  Y_u:=\Sig^{\leq N}_{s,u}(\gamma) = \proj{N}(\Gamma_s^{-1} \Gamma_u),
  \qquad u\in[s,t],
\]
is  tree-like  in \(T^{(N)}(\R^d)\), factoring over this loop $ \tau|_{[s,t]}$. We must prove that every component of level $N+1$ from $\Sig_{s,t}(\gamma)$ vanishes, too.

Let $ w=i_1\cdots i_N$ be a word of length $N$, and let $ j\in\{1,\ldots,d\}$ be a letter. Consider the 
two-dimensional coordinate projection of the lifted and restricted loop $Y$ via 
\[
  P_{w,j}:T^{(N)}(\R^d)\to\R^2,
  \qquad
  P_{w,j}(a):=(a^w,a^j).
\]
Since $Y$ is tree-like, the planar bounded variation path
\[
  Z_u:=P_{w,j}(Y_u)
  =\bigl(S^w_{s,u}(\gamma),S^j_{s,u}(\gamma)\bigr)
  =\bigl(S^w_{s,u}(\gamma),\gamma^j(u)-\gamma^j(s)\bigr) \in \R^2
\]
is tree-like, too,  factoring over the continuous loop $ \tau|_{[s,t]}$. By   Lemma~\ref{lem:2DTreeLikeLemma}, 
its second-order signature vanishes. In particular, its $(1,2)$-component is zero:
\[
  0
  =S^{12}_{s,t}(Z)
  =\int_{s<u<v<t} dZ^1_u\,dZ^2_v
  =\int_s^t Z^1_v\,dZ^2_v = S^{wj}_{s,t}(\gamma) \, .
\]
Since $wj$ is an arbitrary word of length $N+1$, all level $ N+1 $ coordinates of $ \Sig_{s,t}(\gamma) $ vanish. Hence the level $ N+1 $ lift has the property that $ \Sig^{\leq N+1}_{s,t}(\gamma) = \id $, i.e., is trivial. This means that $ \Sig^{\leq N+1}_{0,s}(\gamma) = \Sig^{\leq N+1}_{0,t}(\gamma) $ by Chen's identity.

Thus we have proven that if $\tau(s) = \tau(t)$ for $ s < t $, then  $ \Sig^{\leq N+1}_{0,s}(\gamma) = \Sig^{\leq N+1}_{0,t}(\gamma) $. Therefore, $ t\longmapsto \Sig^{\leq N+1}_{0,t}(\gamma) $ factors over the loop $\tau$, too. The associated map $\psi $ is continuous by Remark \ref{openmap}.
\end{proof}

\section{The sub-interval lemma and the construction of the factoring tree}\label{sec:subinterval_lemma}

We next prove the converse direction, i.e., a path with trivial signature is tree-like (see \cite{BGLY2016}, but there in greater generality). The subsequent lemma contains the key idea and states that a path with trivial total signature must contain a nontrivial subpath, i.e., a restriction on a nontrivial sub-interval $[s,t] \subset [0,1]$, where the total signature is trivial, too.

\begin{lemma}[Sub-interval lemma]\label{lem:subinterval}
Let $\gamma:[0,1]\to\R^d$ be continuous of bounded variation, and suppose
\[
  \Sig_{0,1}(\gamma)=\id \, .
\]
Then there exist \(0\leq s<t\leq 1\) with $ t-s < 1$, i.e., a proper sub-interval, such that
\[
  \Sig_{s,t}(\gamma)=\id.
\]
\end{lemma}

\begin{proof}
Assume, by contradiction, that there is no proper sub-interval with trivial signature. Set
\[
  \Gamma_t:=\Sig_{0,t}(\gamma).
\]
Then $\Gamma_0=\Gamma_1 = \id $ in view of  $ \Sig_{0,1}(\gamma)=\id $. Moreover, if $\Gamma_s=\Gamma_t$ for $ s < t$ and $t-s <1$, then by  Chen's identity,
\[
  \Sig_{s,t}(\gamma)=\Gamma_s^{-1}\Gamma_t= \id \, .
\]
By assumption, this cannot happen. Hence $\Gamma$ defines a simple closed loop, i.e., an injective map,
\[
  \Gamma:S^1\to T((\R^d)), 
\]
where $S^1 $ denotes the $1$-sphere in $\R^2$. Since $S^1$ is compact, $\Gamma$ is a homeomorphism onto its image. Let
\[
  h:\Gamma(S^1)\to S^1\subset\R^2
\]
be its inverse.

The coordinate functions of $h=(h_1,h_2)$ are continuous on the compact set $\Gamma(S^1)$. By Proposition \ref{prop:SW}, the two coordinate functions of $h$ can be uniformly approximated 
by the restriction of linear functionals on $\Gamma(S^1)$. Hence there is a continuous linear map
\[
  l:T((\R^d))\to\R^2
\]
such that $l\circ \Gamma: S^1\to \R^2$ 
is uniformly close to the identity map on $S^1$. Choosing the approximation sufficiently close, the planar curve $ l \circ \Gamma $ stays away from the origin and has nonzero winding number:
\[
  \wind_0( l \circ\Gamma)\neq 0.
\]
An illustration is given in Figure~\ref{fig:winding_approximation}. 
The winding number is determined by integrating the angular one-form 
\[
  \omega=\frac{x\,dy-y\,dx}{x^2+y^2}.
\]
Since $ l \circ \Gamma $ remains in a compact subset of $\mathbb{R}^2\setminus\{0\}$, the form $\omega$ can be uniformly approximated there by polynomial one-forms. Applying Proposition \ref{prop:SW}, namely that the linear functionals on $\Gamma(S^1)$ form an algebra, we finally obtain linear functionals $\ell_1^k, \ell_2^k \in T(\mathbb{R}^d)$, such that the approximate winding number can be written as
\[
\sum_k  \int_0^1 \ell_1^k(\Gamma_t) \, d \ell_2^k(\Gamma_t) \neq 0 \, .
\]

By Lemma~\ref{lem:polynomial-one-form}, these integrals, however, are a finite linear combination of total signature coordinates of $\gamma$ for nonempty words, hence are determined by $\Sig_{0,1}(\gamma)$. Since $ \Sig_{0,1}(\gamma)=\id $, all signature components for nonempty words vanish, whence the integrals have to vanish, too, which is the desired contradiction.

Therefore, a proper sub-interval $ [s,t] $ with $ \Sig_{s,t}(\gamma)=\id $ must exist.
\end{proof}

\subsection*{Factoring Tree Construction}

Recall the signature range  
$$\Gamma([0,1]) = \{\Gamma_t, t\in [0,1]\} \subset T((\R^d))$$ introduced in Proposition~\ref{prop:SW}. 

\begin{figure}[t]
    \centering
    \begin{tikzpicture}[scale=1.05,
        ->-/.style={postaction={decorate,decoration={
            markings,
            mark=at position #1 with {\arrow[scale=1.5]{stealth}}
        }}}
    ]
        \fill (0,0) circle (2pt) node[below right] {$0$};

        \draw[->-={0.135}, thick, black, densely dashed] (0,0) circle (2cm);
        \node[black, right] at (1.75,1.05) {$S^1$};

        \draw[
            ->-={0.383},
            thick,
            black,
            domain=0:360,
            samples=300,
            variable=\t,
            line join=round
        ]
        plot (
            {\t + ((\t>36)*(\t<84)
                 + (\t>180)*(\t<216)
                 + (\t>240)*(\t<288))
                 *10*sin(15*\t)}
            :
            {2 + 0.14*sin(10*\t) + 0.13*cos(15*\t)}
        );

        \node[black, above left] at (-1.8,0.99) {$l \circ \Gamma$};
    \end{tikzpicture}
    \caption{Map $l \circ \Gamma$ (solid) approximating the standard parametrization
    of $S^1$ (dashed) in the proof of Lemma~\ref{lem:subinterval} and preserving
    the non-zero winding number around the origin.}
    \label{fig:winding_approximation}
\end{figure}
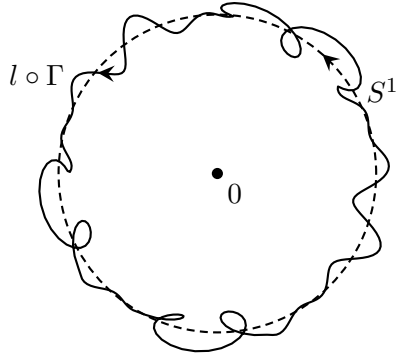

\begin{theorem}[Existence of a factoring tree]\label{thm:quotient-tree}
Let $ \gamma:[0,1]\to\R^d $ be a continuous path of bounded variation such that $ \Sig_{0,1}(\gamma)=\id $. Then, 
$$ T_\gamma := \Gamma([0,1]) \subset T((\R^d))$$
is a compact tree. Moreover, there is a continuous map $  \psi:T_\gamma\to\R^d$ such that 
  $\gamma=\psi\circ \Gamma$. 
\end{theorem}
\begin{proof}
Let $a,b \in T_\gamma$ with $a \neq b$. We have to show that there is a unique (up to reparametrization) continuous injective arc connecting $a$ to $b$. If this were not the case, there would exist two distinct arcs $\alpha$ and $\beta$ connecting $a$ with $b$. Since they are different, there are proper subarcs $\alpha|_{[u_-,u_+]}$ and $\beta|_{[v_-,v_+]}$ such that $\alpha(u_-)=\beta(v_-)$, $\alpha(u_+)=\beta(v_+)$, but otherwise disjoint. Concatenating those subarcs leads to an injective loop in $T_\gamma$, which contradicts the sub-interval lemma. Therefore, $T_\gamma$ is a compact tree. Hence $\gamma$ is tree-like. 

 For the second assertion, we may assume without loss of generality that $\gamma(0) = 0$. We then define $\psi:T((\R^d))\to\R^d$ as the canonical projection of the signature onto its first level, i.e.,  the curve itself. 
\end{proof}

\section{Hambly-Lyons uniqueness theorem}

\begin{theorem}[Hambly--Lyons uniqueness]\label{thm:main}
Let $ \gamma:[0,1]\to\R^d $ be a continuous bounded variation path. Then
\[
  \Sig_{0,1}(\gamma)=\id
  \quad\Longleftrightarrow\quad
  \gamma \text{ is tree-like}.
\]
Consequently, the signature separates based continuous bounded variation paths up to tree-like equivalence, in the sense of Hambly--Lyons \cite{HL2010}.
\end{theorem}

\begin{proof}
Assume first that $\gamma$ is tree-like. By Proposition~\ref{prop:lifts-tree-like}, for every $N\geq1$,  the lifted path
\[
  t\longmapsto \Sig^{\leq N}_{0,t}(\gamma)
\]
is tree-like and closed. Hence its endpoint agrees with its starting point, and therefore
\[
  \Sig^{\leq N}_{0,1}(\gamma)=\id
\]
for every $N$. Thus $\Sig_{0,1}(\gamma)=\id$.

Conversely, assume $\Sig_{0,1}(\gamma)=\id$. By Theorem~\ref{thm:quotient-tree}, $\gamma$ factors as
\[
  \gamma=\psi \circ \Gamma, 
\]
where $T_\gamma =  \Gamma([0,1])$ is a compact tree. 
Hence $\gamma$ is tree-like.
\end{proof}

\begin{remark}
The sub-interval lemma and its consequences carry over verbatim to the rough paths case leading to an elementary proof of the so-called difficult direction from trivial signature to tree-like in \cite{BGLY2016}. Notice that we use here the classical notion of winding numbers, which applies to  \emph{continuous} loops. 

To prove the other direction, namely that tree-like rough paths (see \cite{BGLY2016} for a definition) actually have trivial signatures, one cannot use our approach as outlined in Lemma \ref{lem:2DTreeLikeLemma}, since the Cufí--Verdera  formula  \cite[Theorem 1]{CV2015} does not hold beyond the Young setting. In the rough case, the proof presented in \cite{BGLY2016}, which relies on approximations by finite trees, is the easiest so far.
\end{remark}

\noindent \textbf{Disclaimer and Acknowledgement.} This research was initiated during Valentin Tissot-Daguette's PhD studies at Princeton University and continued during a visit of Walter Schachermayer at ETH Zurich. The research was partially funded by the Austrian Science Fund (FWF) 10.55776/P35519. AI was used for typesetting and language clean-up.

\end{document}